\documentclass[12pt]{amsart}
\usepackage{amsmath,graphicx,amsthm,mathrsfs}
\usepackage[pdftex]{hyperref}
\usepackage{subfig,verbatim}
\usepackage[margin=1in,centering]{geometry}

\newtheorem{prop}{Proposition}
\newtheorem{lemma}{Lemma}
\newtheorem{theorem}{Theorem}

\theoremstyle{definition}
\newtheorem*{rem}{Remark}

\title{A note on Khovanov-Rozansky $sl_2$-homology and ordinary Khovanov homology}
\author{Mark C. Hughes}

\begin{document}

\maketitle

\begin{abstract}
In this note we present an explicit isomorphism between Khovanov-Rozansky $sl_2$-homology and ordinary Khovanov homology.  This result was originally stated in Khovanov and Rozansky's paper \cite{KRI}, though the details have yet to appear in the literature.  The main missing detail is providing a coherent choice of signs when identifying variables in the $sl_2$-homology.  Along with the behavior of the signs and local orientations in the $sl_2$-homology, both theories behave differently when we try to extend their definitions to virtual links, which seemed to suggest that the $sl_2$-homology may instead correspond to a different variant of Khovanov homology.  In this paper we describe both theories and prove that they are in fact isomorphic by showing that a coherent choice of signs can be made.  In doing so we emphasize the interpretation of the $sl_2$-complex as a cube of resolutions.  
\end{abstract}
 
\section{Introduction}
In \cite{KRI} Khovanov and Rozansky constructed for each $n \geq 2$ a homology theory for oriented links in $\mathbb{R}^3$, which assigns to each such link $L$ a bigraded $\mathbb{Q}$-vector space $\mathcal{H}_n(L)$.  The graded Euler characteristic of this space yields the $sl_n$-polynomial of $L$.  They claimed that when $n=2$, this theory is equivalent to the homology theory defined earlier by Khovanov in \cite{khovanov}, which similarly categorifies the $sl_2$-polynomial (Jones polynomial).  Khovanov's original homology theory can also be constructed over $\mathbb{Q}$, in which case it assigns to each oriented link $L$ a bigraded $\mathbb{Q}$-vector space $\mathcal{H}(L)$.  In both setups one of the two gradings present is homological, i.e. arising from the chain complex structures used to define $\mathcal{H}(L)$ and $\mathcal{H}_n(L)$, while the other is induced from an additional grading on the chain complexes and is referred to as the quantum grading.   For each oriented link $L$, let $\mathcal{H}^{i,j}(L)$ and $\mathcal{H}^{i,j}_n(L)$ denote the subspaces of $\mathcal{H}(L)$ and $\mathcal{H}_n(L)$ consisting of elements with homological grading $i$ and quantum grading $j$.  

Although the equivalence between $\mathcal{H}$ and $\mathcal{H}_2$ is often referred to (see e.g., \cite{KRI,Rasmussen}), the details have yet to appear in the literature.  In this note we provide these details, in proving the following:

\begin{theorem}
\label{maintheorem}
For any oriented link $L$, $\mathcal{H}^{i,j}(L)$ and $\mathcal{H}^{i,-j}_2 (L)$ are isomorphic as $\mathbb{Q}$-vector spaces for any $i,j \in \mathbb{Z}$.  
\end{theorem}

Note that our sign convention for crossings will be opposite those in \cite{KRI} (see Figure \ref{fig:resolutions} below), and hence the above correspondence does not require taking the mirror image of $L$ like the statement in \cite{KRI}.

There are a few reasons why providing the details of this isomorphism may be important.  First, $\mathcal{H}_2 (L)$ is defined using polynomial rings $\mathbb{Q}[z_1,\ldots,z_m]$, where the $z_j$ correspond to marked points on a diagram of $L$.  Nearby variables can be identified with signs corresponding to local orientations (see Section \ref{sec:computingU2}).  The presence of these signs and local orientations seems to suggest a similarity with Clark, Morrison and Walker's disoriented Khovanov homology in \cite{disoriented}.  Furthermore, both the $sl_2$ and ordinary theories behave very differently when their definitions are extended to virtual links.  The definition of the $sl_2$-theory extends with no additional difficulties, while the ordinary theory requires considerable effort (see \cite{Manturov}).  On the surface these details suggest that $\mathcal{H}_2$ may instead correspond to a different variant of Khovanov homology other than the ordinary theory $\mathcal{H}$.  However, as we show below, choices can indeed be made so that the signed identification of variables do not cause a problem when constructing the isomorphism, and that the correspondence actually breaks down when we allow virtual crossings.

We begin by briefly outlining the constructions of $\mathcal{H}(L)$ and $\mathcal{H}_2 (L)$, both of which start with a diagram $D$ of $L$, from which we construct a ``cube of resolutions".  We denote the ordinary Khovanov cube of resolutions by $U$, and the $sl_2$-Khovanov-Rozansky cube of resolutions by $U_2$.  The constructions of $U$ and $U_2$ are outlined in Sections \ref{sec:ordinary} and \ref{sec:KRsl2} respectively.  Also contained in Section \ref{sec:KRsl2} are brief descriptions of both matrix factorizations and trivalent graphs, as well as some standard matrix factorization lemmas which will be needed later on.  Section \ref{sec:computingU2} contains computations which give explicit descriptions of the vector spaces and maps associated to the vertices and edges of $U_2$.  The reader who is familiar with the constructions of $\mathcal{H}(L)$ and $\mathcal{H}_2(L)$ is invited to skip directly to Section \ref{sec:identifyingcubes}, where the isomorphisms of Theorem \ref{maintheorem} are constructed.  These isomorphisms are described on the level of the cubes $U$ and $U_2$.  This description begins in Section \ref{subsec:underlyingcubes} by identifying the underlying 1-skeleta of the cubes.  In Sections \ref{subsec:signs} and \ref{subsec:generators} we then fix consistent choices of signs and generators for the spaces at the vertices of the cube $U_2$.  In Section \ref{subsec:isomorphisms} we construct the required isomorphisms between the spaces assigned to each vertex by $U$ and $U_2$, and verify that these isomorphisms commute with the edge maps.  The homological and quantum gradings of these maps are then computed in Sections \ref{subsec:homologicalgradings} and \ref{subsec:quantumgradings}. 

\subsection*{Acknowledgements} The author would like to thank Oleg Viro for many helpful discussions.

\section{Ordinary Khovanov homology} \label{sec:ordinary} 
\begin{figure}
 \centering
 \includegraphics[width=10cm]{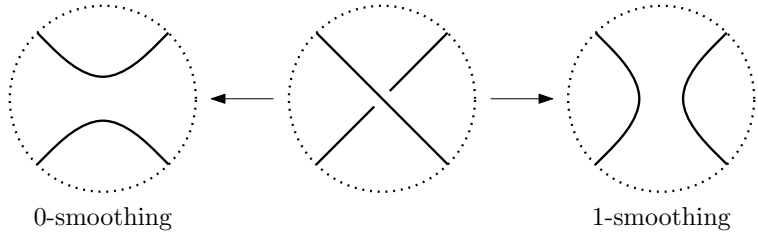}
  \caption{Two smoothings of a crossing}
  \label{fig:smoothings}
\end{figure}

Given a link diagram $D$, we can smooth each crossing in two different ways, called the $0$ and $1$-smoothings (see Figure \ref{fig:smoothings}).  Given an ordering of the $n$ crossings of $D$, to each vertex $v \in \{0,1\}^n$ of the $n$-dimensional cube we can associate a total smoothing $D_v$, where the $j$-th crossing of $D$ is smoothed according to the $j$-coordinate of $v$.  For each such $v$ the total smoothing $D_v$ will consist of a collection of disjoint circles in the plane.  We let $s(v)$ denote the number of circles making up $D_v$, and $r(v)$ the number of 1-smoothings performed to obtain $D_v$.  Let $V$ denote the graded $\mathbb{Q}$-vector space spanned by two elements $\{1,x\}$, where $\text{deg } 1 = 1$ and $\text{deg } x = -1$.  If $n_+$ and $n_-$ denote the number of positive and negative crossing of $D$ respectively, then to each $D_v$ we associate the vector space $\mathbb{V}_v=V^{\otimes s(v)}\{ r(v)+n_+-2n_-\}$ (where $W\{k\}$ denotes the vector space $W$ with grading shifted by $k$).  Fix a one-to-one correspondence between the circles in $D_v$ and the tensor product factors of $\mathbb{V}_v$.

Suppose that $v$ and $v'$ are two vertices which differ in one coordinate only, with $v_j = 0$, $v'_j = 1$, and $v_i =v'_i$ for all $i\neq j$.  Then to the edge $e$ connecting $v$ and $v'$ we associate a map $\varphi_e:\mathbb{V}_v \rightarrow \mathbb{V}_{v'}$.  Suppose first that changing the smoothing of the $j$-th crossing from 0 to 1 corresponds to the merging of two circles in $D_v$.  Then on the tensor product factors of $\mathbb{V}_v$ corresponding to these two circles, $\varphi_v$ is defined as
\begin{alignat*}{2}
x \otimes x &\mapsto 0 \qquad 1 \otimes 1 \mapsto 1\\
x \otimes 1 & \mapsto x \qquad  1 \otimes x \mapsto x.
\end{alignat*}
If the smoothing change corresponds to splitting a circle of $D_v$ into two, then on the affected tensor product factor of $\mathbb{V}_v$ we define $\varphi_e$ as 
\begin{alignat*}{2}
x &\mapsto x \otimes x \qquad 1 \mapsto 1 \otimes x + x \otimes 1,
\end{alignat*}
while $\varphi_e$ is the identity on all other tensor product factors of $\mathbb{V}_v$.  It can be shown that the edge maps $\varphi_e$ commute around any 2-dimensional face of $\left[0,1\right]^n$.  Let $U$ denote this cube $\left[0,1\right]^n$ with the assignments $v \mapsto \mathbb{V}_v$ and $e \mapsto \varphi_e$ for each vertex $v$ and edge $e$.

Now define 
\[
C^k(D) = \bigoplus_{r(v)= k + n_-}\mathbb{V}_v,
\]
and let $d^k : C^k(D) \rightarrow C^{k+1}(D)$ be defined by 
\[
d^k = \sum \epsilon_e \varphi_e,
\]
where the sum is over all edges $e$ between vertices $v$ and $v'$ with $r(v') = r(v)+1=k+n_-+1$.  Each $\epsilon_e = \pm 1$, and is chosen so that $d^{k+1} d^k = 0$ for all $k$.  We say that any element of $C^k(D)$ has homological degree $k$.  Note that $C^k(D)$ also inherits a grading from the $\mathbb{V}_v$ which we call the quantum grading, and it is easy to check that $d^k$ has degree 0 with respect to this grading.   

Equipped with this differential, the $C^j(D)$ form a chain complex, which we denote by $\mathscr{C}(D)$.  The homology of $\mathscr{C}(D)$ is the ordinary Khovanov homology $\mathcal{H}(L)$ defined first in \cite{khovanov}.  It is a bigraded vector space, one grading coming from the homological grading of the chain complex $C^*(D)$, the other is inherited from the quantum grading of $\mathbb{V}_v$.  As indicated by the notation, up to isomorphism $\mathcal{H}(L)$ is independent of the choice of diagram $D$.  

\section{Khovanov-Rozansky $sl_2$-homology} \label{sec:KRsl2}  The definition of the Khovanov-Rozansky $sl_n$-homology $\mathcal{H}_n (L)$ is more involved.  Since we are only concerned with the $n=2$ case, for concreteness we will only outline the construction of $\mathcal{H}_2 (L)$.  For $n>2$ the construction is similar, the details of which can be found in \cite{KRI,WuKR}.  We begin by describing matrix factorizations, whose homology plays the role of the vector spaces $\mathbb{V}_v$ in this setup. 

\subsection{Matrix factorizations} Let $R$ be a commutative ring with unity.  A matrix factorization with potential $\omega \in R$ is a pair of free $R$-modules $M^0$ and $M^1$ with maps $d^0 : M^0 \rightarrow M^1$ and $d^1 : M^1 \rightarrow M^0$, such that $d^1d^0 = \omega \cdot \text{id}_{M^0}$ and $d^0d^1= \omega \cdot \text{id}_{M^1}$.  We can treat matrix factorizations much as we would a $\mathbb{Z}_2$-graded chain complex 
\[
M^0 \xrightarrow{\quad d^0 \quad  } M^1 \xrightarrow{\quad d^1 \quad} M^0,
\]
and take tensor products and direct sums to yield new matrix factorizations.  In the case of tensor products we use the Leibniz rule, so that the tensor product of factorizations with potential $\omega$ and $\omega'$ will have potential $\omega+\omega'$.  Note that any matrix factorization with $\omega = 0$ is a cyclic chain complex, and hence we can take its homology.  

A morphism $f$ from the matrix factorization $M=\{M^0 \rightarrow M^1 \rightarrow M^0\}$ to $N=\{N^0 \rightarrow N^1 \rightarrow N^0\}$ consists of a pair of maps $f^0 : M^0 \rightarrow N^0$ and $f^1 :M^1 \rightarrow N^1$ which commute with the differentials on $M$ and $N$.  If $M$ and $N$ both have potential 0, then we say that $f$ is a \emph{quasi-isomorphism} if there is a morphism $g$ from $N$ to $M$ such that both $f \circ g$ and $g \circ f$ are chain homotopic to the identity morphisms.  In this case $f$ and $g$ induce isomorphisms on the homology of $M$ and $N$, and we write $M \cong N$.  While more general notions of equivalence exist for matrix factorizations, this will suffice for our purposes.   

For a matrix factorization $M=\{M^0 \rightarrow M^1 \rightarrow M^0\}$ with $M^0$ and $M^1$ of rank 1, we let $|0\rangle$ and $|1\rangle$ denote the generators of $M^0$ and $M^1$ respectively.  We extend this notation to tensor products of such factorizations.  For example, if $N=\{N^0 \rightarrow N^1 \rightarrow N^0\}$ is another such matrix factorization, then $M\otimes N$ has basis $\{|00\rangle,|11\rangle,|01\rangle,|10\rangle \}$, where $M^0 \otimes N^1$ is generated by $|01\rangle$, $M^1 \otimes N^1$ is generated by $|11\rangle$, etc.

For $a , b \in R$, let $(a,b)_R$ denote the matrix factorization with potential $ab$ given by $R \xrightarrow{\text{  } a \text{  }} R \xrightarrow{\text{  } b \text{  }} R$ (where the differentials are given by multiplication by $a$ and $b$ respectively).  For $a_1,\ldots,a_k,b_1,\ldots,b_k \in R$, let 
\[ 
\left( \begin{array}{cc}
a_1 & b_1  \\
\vdots & \vdots  \\
a_k & b_k  \end{array} \right)_R
\]
denote the tensor product of the $(a_1,b_1)_R, \ldots , (a_k,b_k)_R$.  We will often drop the ring $R$ from the notation when there is no risk of confusion. 

Now let $R=\mathbb{Q}\left[z_1,\ldots,z_m\right]$, $R'=\mathbb{Q}\left[z_2,\ldots,z_m\right]$, and suppose that $M$ is a matrix factorization with potential 0 given by 
\begin{equation}
\label{eqn:eliminationofvariable} 
\left( \begin{array}{cc}
a_1 & b_1  \\
\vdots & \vdots  \\
a_k & b_k  \end{array} \right)_R
\end{equation}
where $a_1,\ldots,a_k,b_1,\ldots,b_k \in R$.  Suppose for some $j$ that $b_j = z_1 - c$, where $c \in R'$.  Let $\pi : R \rightarrow R'$ be the $\mathbb{Q}$-algebra homomorphism with $\pi (z_i) = z_i$ for $i \geq 2$, and $\pi(z_1)=c$.  Consider the factorization $M'$ obtained by dropping the row $(a_j,b_j)$ from (\ref{eqn:eliminationofvariable}), and applying $\pi$ to all the other entries.  Note that both $M$ and $M'$ have potential 0 and can be thought of as matrix factorizations over the ring $R'$.

\begin{lemma}
\label{lem:eliminationofvariables}
$M$ and $M'$ are quasi-isomorphic as matrix factorizations over $R'$.
\end{lemma}

\begin{proof}
Recall that $M$ is a free $R$-module with basis consisting of all $|\alpha \rangle$, where $\alpha$ is a string of length $k$ in 0s and 1s.  Let $\beta$ be any string of 0s and 1s of length $k-1$, and $u,v \in R$.  The map $M \rightarrow M'$ sending
\begin{alignat*}{2}
u \cdot | 0 \beta \rangle  &\mapsto \pi (u) \cdot | \beta \rangle\\
v \cdot | 1 \beta \rangle  & \mapsto 0
\end{alignat*}
is a quasi-isomorphism (see \cite{KRI,Rasmussen} for details).
\end{proof}

Note that the lemma also holds with the roles of $a_j$ and $b_j$ reversed, by reversing the roles of 0 and 1 in the isomorphism and shifting the $\mathbb{Z}_2$-matrix factorization grading.  In what follows we will make repeated use of Lemma \ref{lem:eliminationofvariables} to simplify different matrix factorizations.  For notational convenience however, we will avoid renaming the ground ring each time we do this, and instead think of this procedure as dropping a variable from the original ground ring $R$.      

We will also need the following lemma, whose proof is in \cite{KRI}:

\begin{lemma}
For any invertible $c \in R$, $M$ is quasi-isomorphic to a matrix factorization obtained by replacing the row $(a_j,b_j)$ of (\ref{eqn:eliminationofvariable}) with $(ca_j,c^{-1}b_j)$.
\end{lemma}

\subsection{Marked trivalent graphs} We now describe a class of planar graphs we will need during the construction.  These graphs consist of two types of edges, regular edges and wide edges.  Each vertex is either incident to three edges, two of which are regular with the other one wide, or
incident to a single regular edge.  We also allow disjoint simple loops in our graphs, which we think of as regular edges with no starting or ending vertex.  Regular edges are oriented, and each wide edge shares one endpoint with two outgoing regular edges and the other endpoint with two incoming regular edges.

Finally, we add a finite collection of marked points $\{z_i\}$ to the regular edges, so that each regular edge has at least one marked point on it, and so that the univalent vertices are all marked (see Figure \ref{fig:trivalent}, where we denote the marked points on regular edges with arrow heads which also indicate the orientation).  We will refer to such an object as a \emph{marked trivalent graph}.  We call such a graph \emph{closed} if it has no vertices of degree 1.
\begin{figure}[h]
 \centering
 \includegraphics[width=9cm]{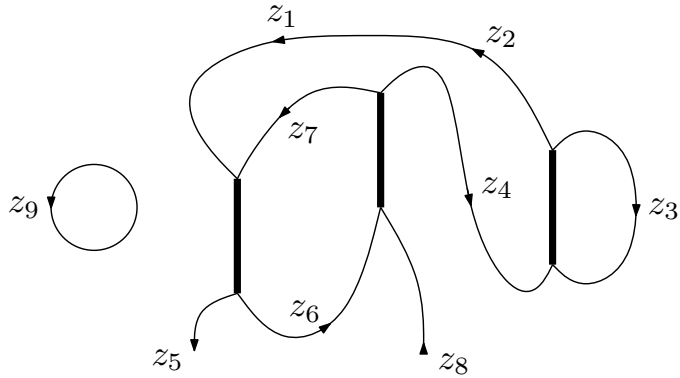}
  \caption{A marked trivalent graph}
  \label{fig:trivalent}
\end{figure} 

\subsection{Trivalent graphs and graded matrix factorizations} Let $\Gamma$ be a marked trivalent graph with marked points $\{z_1, \ldots, z_m\}$, and let $R = \mathbb{Q} \left[ z_1 , \ldots , z_m \right]$.  We make $R$ into a graded ring by setting $\text{deg }z_i=2$ for each $i$.  

Now if $z_i$ and $z_j$ are adjacent marked points on the same regular edge of $\Gamma$ which is oriented from $z_j$ to $z_i$, and $L^i_j$ is the arc connecting them, then define $C(L^i_j)$ be the matrix factorization
\[
R \xrightarrow{\quad z^2_i+z_iz_j+z^2_j \quad  } R\{-1\} \xrightarrow{\quad z_i-z_j \quad} R.
\]
We also allow the case when $z_i=z_j$, on a loop with a single marked point.  Note with the grading shifts, the differentials are homogenous of degree 3, and that $C_2(L^i_j)$ has potential $z_i^3-z_j^3$.

\begin{figure}
 \centering
 \includegraphics[width=1.75cm]{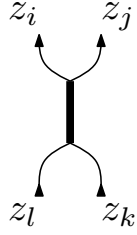}
  \caption{A wide edge with surrounding marked points}
  \label{fig:wideedge}
\end{figure} 

Now let $E^{ij}_{kl}$ be a wide edge of $\Gamma$ with surrounding marked points as in Figure \ref{fig:wideedge}.  Then to $E^{ij}_{kl}$ we associate the matrix factorization $C_2(E^{ij}_{kl})$, which is the tensor product of factorizations
\[
R \xrightarrow{\quad -3 \left(z_k+z_l\right) \quad  } R\{1\} \xrightarrow{\quad z_iz_j-z_kz_l \quad} R,
\]
and
\[
R\{-1\} \xrightarrow{\quad u(z_i,z_j,z_k,z_l) \quad  } R\{-2\} \xrightarrow{\quad z_i+z_j-z_k-z_l \quad} R\{-1\},
\]
where
\[
u(z_i,z_j,z_k,z_l)=z_i\left(z_k+z_l-z_j\right)+z_i^2+z_j\left(z_k+z_l\right)+z_j^2+\left(z_k+z_l\right)^2.
\]
Note again that the differentials are all homogeneous of degree 3, and that $C_2(E^{ij}_{kl})$ has potential $z_i^3+z_j^3-z_k^3-z_l^3$.

Now to the marked trivalent graph $\Gamma$ we associate the matrix factorization  
\[
C_2(\Gamma) = \bigotimes_{E^{ij}_{kl}} C_2(E^{ij}_{kl}) \otimes \bigotimes_{L^i_j}
C_2(L^i_j)
\]
where the first tensor product runs over all wide edges $E^{ij}_{kl}$ of $\Gamma$, and the second runs over all of its oriented arcs $L^i_j$ as described in above.  If $\Gamma$ is closed, then $C_2(\Gamma)$ is a matrix factorization with potential 0, and thus we can take the homology of $C_2(\Gamma)$, which we denote by $H_2(\Gamma)$.  Since the differentials on the tensor factors are all homogeneous, $H_2(\Gamma)$ inherits a $\mathbb{Z}$-grading from the grading on $R$ (which we refer to as the quantum grading), as well as a $\mathbb{Z}_2$-homological grading from the matrix factorization structure.  It turns out however that $H_2(\Gamma)$ is contained entirely in one of the $\mathbb{Z}_2$ degrees, and hence we ignore this grading from now on.  

\subsection{Morphisms $\chi_0$ and $\chi_1$} 
\begin{figure}
 \centering
 \includegraphics[width=7cm]{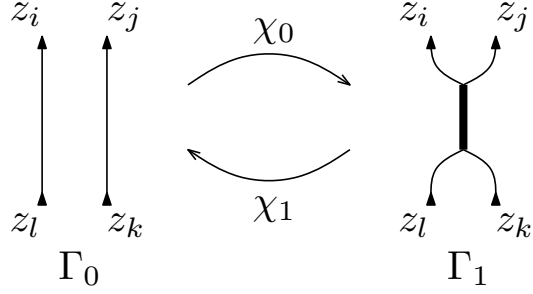}
  \caption{Maps $\chi_0$ and $\chi_1$}
  \label{fig:chimaps}
\end{figure}

Consider now the two graphs $\Gamma_0$ and $\Gamma_1$ in Figure \ref{fig:chimaps}.  Note that $C_2(\Gamma_0)= C_2(L^j_k) \otimes C_2(L^i_l)$ and $C_2(\Gamma_1) = C_2(E^{ij}_{kl})$ are both tensor products of two such matrix factorizations.  For both $i$, order the tensor product basis of $C^0_2(\Gamma_i)$ as $\{|00\rangle, |11\rangle \}$, and the basis of $C^1_2(\Gamma_i)$ as $\{|01\rangle, |10\rangle \}$.  

We now define a morphism $\chi_0 : C_2(\Gamma_0) \rightarrow C_2(\Gamma_1)$, which is described in the above basis by matrices
\[ 
U^0=\left( \begin{array}{cc}
z_i-z_k & 0 \\
z_i-z_j-2 z_k-z_l & 1
\end{array} \right)
\qquad
U^1=\left( \begin{array}{cc}
z_i & -z_k \\
-1 & 1
\end{array} \right).
\]
Likewise we can define a morphism $\chi_1 : C_2(\Gamma_1) \rightarrow C_2(\Gamma_0)$ in the above bases by the pair of matrices
\[ V^0=\left( \begin{array}{cc}
1 & 0 \\
-z_i+z_j+2 z_k+z_l & z_i- z_k
\end{array} \right)
\qquad
V^1=\left( \begin{array}{cc}
1 & z_k \\
1 & z_i
\end{array} \right).
\]
Note that both of these maps raise the quantum grading by 1.

\subsection{The Khovanov-Rozansky complex $\mathscr{C}_2(D)$} We are now ready to describe the complex $\mathscr{C}_2(D)$ for an oriented link diagram $D$.  Begin by placing marks on $D$ so that none of the crossings are marked, and so that each component of $D$ and arc between two crossings has a marked point.  From now on, let $R = \mathbb{Q}[z_1,\ldots,z_m]$, where $z_1,\ldots,z_m$ are the marked points on $D$.  

Now if $L^i_j$ is an arc in the diagram of $D$ between points $z_j$ and $z_i$ with no other marked points or crossings in between, we define $\mathscr{C}_2(L^i_j)$ to be the following complex of matrix factorizations
\[
0 \longrightarrow C_2(L^i_j) \longrightarrow 0,
\]
where we assign $C_2(L^i_j)$ to homological degree 0.

If $P_+$ is a positive crossing, we define the chain complex of matrix factorizations $\mathscr{C}_2(P_+)$ to be 
\begin{equation}
\label{eqn:P+}
0 \longrightarrow C_2(\Gamma_0)\{-1\} \xrightarrow{\quad \chi_0 \quad} C_2(\Gamma_1)\{-2\} \longrightarrow 0,
\end{equation}
while for $P_-$ a negative crossing we define $\mathscr{C}_2(P_-)$ to be 
\begin{equation}
\label{eqn:P-}
0 \longrightarrow C_2(\Gamma_1)\{2\} \xrightarrow{\quad \chi_1 \quad} C_2(\Gamma_0)\{1\} \longrightarrow 0.
\end{equation}
In both cases we assign $C_2(\Gamma_0)$ to homological degree 0.

Now we define $\mathscr{C}_2(D)$ to be the chain complex of matrix factorizations
\begin{equation}
\label{eqn:chaincomplex}
\mathscr{C}_2(D) = \bigotimes_{L^i_j} \mathscr{C}_2(L^i_j) \otimes \bigotimes_{P_+} \mathscr{C}_2 (P_+) \otimes \bigotimes_{P_-} \mathscr{C}_2 (P_-),
\end{equation}
where the product runs over all such arcs $L^i_j$ and crossings $P_{\pm}$.  Since $\mathscr{C}_2(D)$ is made up of direct sums and tensor products of graded matrix factorizations, it inherits the structure of a graded matrix factorization with differential $d_{mf}$ and 0 potential.  Furthermore there is a separate differential $d_{\chi}$ on $\mathscr{C}_2(D)$ coming from the complexes $\mathscr{C}_2(L^i_j)$, $\mathscr{C}_2(P_+)$, $\mathscr{C}_2(P_-)$, and the Leibniz rule in (\ref{eqn:chaincomplex}).  This homological differential commutes with the matrix factorization differential $d_{mf}$, and hence induces a differential on $H(\mathscr{C}_2(D),d_{mf})$.  Thus we can take the homology of $H(\mathscr{C}_2(D),d_{mf})$ under this induced differential $d^*_{\chi}$, which gives the Khovanov-Rozansky $sl_2$-homology $\mathcal{H}_2(L)$ of the link $L$.  In other words, 
\[
\mathcal{H}_2(L) = H(H(\mathscr{C}_2(D),d_{mf}),d^*_{\chi}).
\]
As mentioned above, $\mathcal{H}_2(L)$ inherits both the homological and quantum gradings of $\mathscr{C}_2(D)$, and up to bigrading preserving isomorphism does not depend on our choice of diagram $D$ or the markings $\{z_1, \ldots,z_m\}$.

\subsection{$\mathscr{C}_2(D)$ as a cube of resolutions} \label{subsection:KRcube} We can understand the above construction in terms of a cube of resolutions similar to the case of ordinary Khovanov homology.  In Figure \ref{fig:resolutions} we define 0 and 1-resolutions for a positive crossing $P_+$, and $-1$ and 0-resolutions for a negative crossing $P_-$.
\begin{figure}
 \centering
 \includegraphics[width=8cm]{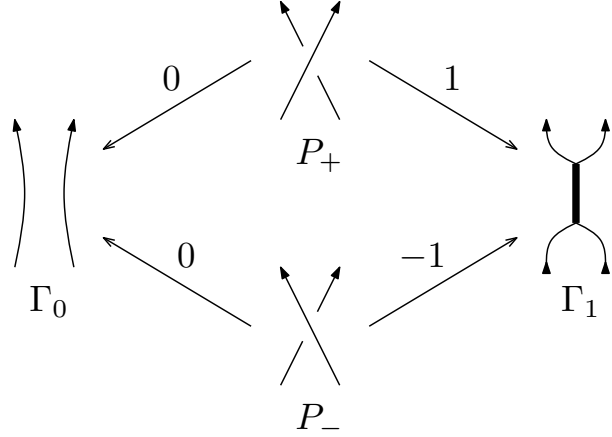}
  \caption{Resolutions of positive and negative crossings of $D$}
  \label{fig:resolutions}
\end{figure}
After fixing an ordering of the crossing of $D$ where all positive crossings are listed first, to each vertex of the cube $\left[0,1\right]^{n_+} \times \left[-1,0\right]^{n_-}$ we can assign a total resolution of $D$ much as we assigned total smoothings to vertices of $\left[ 0,1\right]^n$ in the ordinary case.  With markings on $D$ as above, each total resolution will give a closed marked trivalent graph, which we denote by $\Gamma_v$, where $v$ is the corresponding vertex of $\left[0,1\right]^{n_+} \times \left[-1,0\right]^{n_-}$.

Ignoring the differential $d_{\chi}$ on $\mathscr{C}_2(D)$, we see that $\mathscr{C}_2(D)$ splits as a direct sum of matrix factorizations
\[
\mathscr{C}_2(D) = \bigoplus_v C_2(\Gamma_v)\{p_v\},
\]  
where $v$ runs over all vertices of the cube, and $p_v$ is some integer quantum grading shift.  Indeed, the $R$-submodule $\mathscr{C}^j_2(D)$ sitting in homological degree $j$ of the complex $\mathscr{C}_2(D)$ will itself be a matrix factorization, and will split as 
\[
\mathscr{C}^j_2(D) = \bigoplus_{\sum v_i = j} C_2(\Gamma_v)\{p_v\},
\]
where $v=(v_1,\ldots,v_n) \in \{0,1\}^{n_+} \times \{-1,0\}^{n_-}$.

To each vertex $v$ we assign the matrix factorization $C_2(\Gamma_v)\{p_v\}$.  Suppose that $v=(v_1,\ldots,v_n)$ and $v'=(v'_1,\ldots,v'_n)$ are two vertices which differ in a single coordinate, with $v_j = v'_j-1$, and $v_i=v'_i$ for all $i \neq j$.  Then $v$ and $v'$ are joined by an edge $e$, and we define an edge morphism $\psi_e : C_2 (\Gamma_v)\{p_v\} \rightarrow C_2 (\Gamma_{v'})\{p_{v'}\}$ as follows.  Suppose first that the corresponding crossing is positive, so that $v_j=0$ while $v'_j=1$.  Then the graphs $\Gamma_{v'}$ is obtained from $\Gamma_v$ by making a local replacement of the form in Figure \ref{fig:chimaps} (left-to-right), and    
\[
C_2 (\Gamma_v) = C_2 (L^j_k) \otimes C_2 (L^i_l) \otimes M \quad \text{ and } \quad C_2 (\Gamma_{v'}) = C_2 (E^{ij}_{kl}) \otimes M,
\]
where $M$ is some matrix factorization corresponding to the unchanging crossings and arcs in $\Gamma_v$.  Then define $\psi_e = \chi_0 \otimes \text{id}_M$.  Likewise in the case were $e$ corresponds to a negative crossing with $v_j=-1$ and $v'_j=0$ (right-to-left replacement in Figure \ref{fig:chimaps}), then 
\[
C_2 (\Gamma_v) = C_2 (E^{ij}_{kl}) \otimes N \quad \text{ and } \quad C_2 (\Gamma_{v'}) = C_2 (L^j_k) \otimes C_2 (L^i_l) \otimes N,
\]
and we define $\psi_e = \chi_1 \otimes \text{id}_N$.  Notice that 
\begin{equation}
\label{eqn:edgesigns}
d_{\chi} = \sum_{e} \varepsilon_e \psi_e,
\end{equation}
where the sum is over all edges $e$ of the cube, and the $\varepsilon_e = \pm 1$ are determined by the Leibniz rule in (\ref{eqn:chaincomplex}).  Notice also that the edge maps $\psi_e$ commute around any 2-dimensional face of $\left[0,1\right]^{n_+} \times \left[-1,0\right]^{n_-}$.

After taking the homology of each matrix factorization $C_2(\Gamma_v)\{p_v\}$, we obtain an assignment of graded $R$-modules $H_2(\Gamma_v)\{p_v\}$ to each vertex $v$.  Furthermore, the $\psi_e$ commute with the matrix factorization differentials on $C_2(\Gamma_v)\{p_v\}$ and $C_2(\Gamma_{v'})\{p_{v'}\}$, and hence induce maps $H_2(\Gamma_v)\{p_v\} \rightarrow H_2(\Gamma_{v'})\{p_{v'}\}$, which we will also denote as $\psi_e$.  We therefore let $U_2$ denote the cube $\left[0,1\right]^{n_+} \times \left[-1,0\right]^{n_-}$ with the assignments $v \mapsto H_2(\Gamma_v)\{p_v\}$ and $e \mapsto\psi_e$.  


\section{The Khovanov-Rozansky cube $U_2$}
\label{sec:computingU2}

In order to prove Theorem \ref{maintheorem} we will need explicit descriptions of the spaces $H_2(\Gamma_v)\{p_v\}$ and maps $\psi_e$ associated to the vertices and edges of $U_2$, which we compute in this section.  

\subsection{Computing $H_2(\Gamma)$}  Suppose $\Gamma=\Gamma_v$ has marked points $\{ z_1,\ldots,z_m\}$. Consider the graph $\widetilde{\Gamma}$ obtained by removing all of the wide edges of $\Gamma$.  Then $\widetilde{\Gamma}$ will consist of a collection of $k$ disjoint loops in the plane, each with at least one marked point.  Suppose that the $z_j$ are numbered so that each of the first $k$ points $\{z_1, \ldots ,z_k\}$ lie on a different component of $\widetilde{\Gamma}$.    

\begin{lemma}
\label{lem:matrixfactorizationstructure}
$H_2 (\Gamma)$ is isomorphic as a graded vector space to $\mathbb{Q}\left[z_1,\ldots ,z_k\right] / (z_1^2, \ldots , z_k^2)$, where the $q$-grading is shifted so that $\text{deg }1=-k$ and each $z_i$ has $\text{deg }z_i = 2$.
\end{lemma}


\begin{proof}
Note first that if any regular edge of $\Gamma$ contains more than one marked point, say $z_i$ and $z_j$, then it contributes a factor of 
\[
C_2 (L^i_j) = \left\{ R \xrightarrow{\quad z^2_i+z_iz_j+z^2_j \quad  } R\{-1\} \xrightarrow{\quad z_i-z_j \quad} R\right\},
\]
to $C_2 (\Gamma)$.  By Lemma \ref{lem:eliminationofvariables} we can eliminate this factor by replacing $z_j$ with $z_i$ everywhere in $C_2 (\Gamma)$, and dropping the variable $z_j$ from the base ring $R$.  This elimination isomorphism is graded of degree 0.  We think of this graphically as removing the marked point $z_j$.  

We can also reverse this procedure, and for convenience will assume that each regular edge of $\Gamma$ has more than one marked point.

We will simplify $C_2 (\Gamma)$ by first eliminating the factors associated to its wide edges one at a time.  Let $C_2 (E^{ij}_{kl})$ be such a factor, where because of our assumption on regular edges we can assume that $z_i,z_j,z_k$ and $z_l$ are all distinct. Then (ignoring $q$-grading shifts momentarily) we have
\[
\left( \begin{array}{cc}
-3 \left(z_k+z_l\right) & z_i z_j-z_k z_l  \\
u(z_i,z_j,z_k,z_l) & z_i+z_j-z_k-z_l \end{array} \right) \cong \left( \begin{array}{cc}
u(z_i,z_j,z_k,-z_k) & z_i+z_j \end{array} \right) \langle 1 \rangle,
\]
where $\langle 1 \rangle$ denotes the matrix factorization with the $\mathbb{Z}_2$-matrix factorization gradings of the modules and differentials shifted by 1.  The first row is eliminated while replacing $z_l$ with $-z_k$, and the remaining row can be eliminated while replacing $z_j$ with $-z_i$, with both replaced variables being dropped from the ring.  The elimination isomorphism of the first row is graded of degree $-1$, while the second is graded of degree $+1$.

Thus if $M$ denotes the matrix factorization made up of all the factors of $C_2 (\Gamma)$ except $C_2(E^{ij}_{kl})$, then 
\[
C_2 (\Gamma) = C_2(E^{ij}_{kl}) \otimes M \cong M\langle 1 \rangle,
\]
where the quasi-isomorphism preserves the $q$-grading and is over the ring polynomial with $z_l$ and $z_j$ dropped.  We think of this isomorphism graphically as removing the wide edge $E^{ij}_{kl}$, while relabeling marked points as in Figure \ref{fig:removewideedge}.
\begin{figure}
 \centering
 \includegraphics[width=5cm]{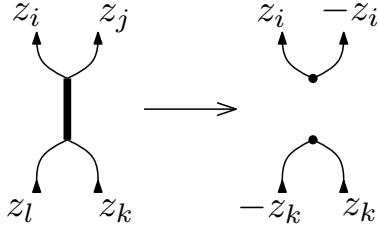}
  \caption{Removal of a wide edge $E^{ij}_{kl}$ while identifying marked points}
  \label{fig:removewideedge}
\end{figure}

\begin{figure}
 \centering
 \includegraphics[width=7cm]{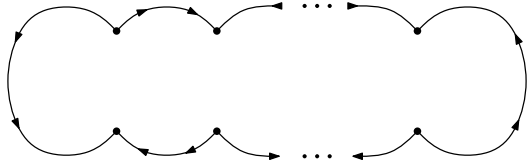}
  \caption{Component of $\widetilde{\Gamma}$ after wide edge removal}
  \label{fig:nowideedges}
\end{figure}

After removing all wide edges from $\Gamma$ we are left with the graph $\widetilde{\Gamma}$.  Each component will consist of oriented regular edges with marked points, as well as vertices left over after the removal of the wide edges.  Each such vertex will be incident to two regular edges, both of which will be oriented either inwards or outwards.  Furthermore the labels on the two marked points nearest a vertex on either side will differ by a sign.

The remaining factors in $C_2(\Gamma)$ will all come from the arcs $L^i_j$ between adjacent marked points on a single regular edge.  Letting $T$ be a component of $\widetilde{\Gamma}$, we use Lemma \ref{lem:eliminationofvariables} as above to eliminate all but a single factor of $C_2(L^i_j)$ coming from the regular edges of $T$.  Graphically this corresponds to removing all but a single marked point from each regular edge of $T$, except for a single edge on which we leave two marked points.

Let $z_i$ and $z_j$ be the markings on the last remaining regular $e$ edge of $T$ which has more than one marking (oriented from $z_j$ to $z_i$).  Now suppose we leave $z_i$ along $e$ in the positively oriented direction and cross the nearest vertex.  This will bring us to a regular edge whose orientation disagrees with that of $e$, and whose marked point is labelled with $-z_i$.  Crossing the next vertex takes us to an edge whose orientation agrees with that of $e$ and whose marked point is labelled with $z_i$.  Continuing on until we return to the marked point $z_j$ on $e$, we see that $z_i = z_j$.  Thus the lone remaining factor from $T$ in $C_2 (\Gamma)$ is of the form 
\[
C_2 (L^i_i) = \left\{ R \xrightarrow{\quad 3z^2_i \quad  } R\{-1\} \xrightarrow{\quad 0 \quad} R\right\},
\]
and if the marked points are labeled so that one of the points $\{z_1 , \ldots , z_k\}$ lives on each of the $k$ components of $\widetilde{\Gamma}$, then 
\begin{equation}
\label{eqn:simplifiedmatrixfactorization}
C_2 (\Gamma) \cong \bigotimes_{i=1}^k C_2 (L^i_i),
\end{equation}
where the quasi-isomorphism preserves the $q$-grading.  It thus follows that 
\[
H_2 (\Gamma) \cong \bigotimes_{i=1}^k H_2 (L^i_i) \cong \left(\mathbb{Q} \left[z_1, \dots , z_k\right]/(z_1^2, \ldots , z_k^2)\right) \{-k\},
\]
where $1 \in \mathbb{Q}\left[z_1, \dots , z_k\right]/(z_1^2, \ldots , z_k^2)$ corresponds to $|11\cdots 1 \rangle$ in the tensor product basis of (\ref{eqn:simplifiedmatrixfactorization}), and has grading $-k$.
\end{proof}

Thus we see that for each vertex $v$ of $U_2$ we have $H_2(\Gamma_v) \cong \mathbb{Q}[z_{i_1},\ldots ,z_{i_j}]/(z_{i_1}^2,\ldots ,z_{i_j}^2)$ for some subset $\{z_{i_1},\ldots ,z_{i_j}\}$ of the marked points $\{z_1,\ldots,z_m\}$ on $D$.  For each vertex $v$, let $R_v = \mathbb{Q}[z_{i_1},\ldots ,z_{i_j}]$ denote the polynomial ring over this subset of marked points.

\subsection{Computing the edge maps $\psi_e$} \label{sec:computingedgemaps} Using Lemma \ref{lem:eliminationofvariables} we can also rewrite the edge maps $\psi_e: H_2(\Gamma_v) \rightarrow H_2(\Gamma_{v'})$ in a simpler form.  Indeed, since each edge corresponds to changing the resolution at a single crossing, we can use the lemma to eliminate all the factors of $C_2(\Gamma_v)$ corresponding to wide edges and extra regular arcs away from the affected crossing.  This will result in relabeling marked points on the graph as variables are identified in the polynomial ring.  The edge maps $\psi_e$ are then given by $\widehat{\chi}_i \otimes \text{id}_M$ similar to before, where $\widehat{\chi}_i$ is given by making the same variable identifications in the matrix description of $\chi_i$, and $M$ denotes the product of all the unaffected factors.  Furthermore, since components of $\widetilde{\Gamma}_v$ and $\widetilde{\Gamma}_{v'}$ correspond to tensor product factors in the homology, we can safely ignore all but the affected components, since $\psi_e$ will be the identity on all these factors. 
\begin{figure}
 \centering
 \includegraphics[width=11cm]{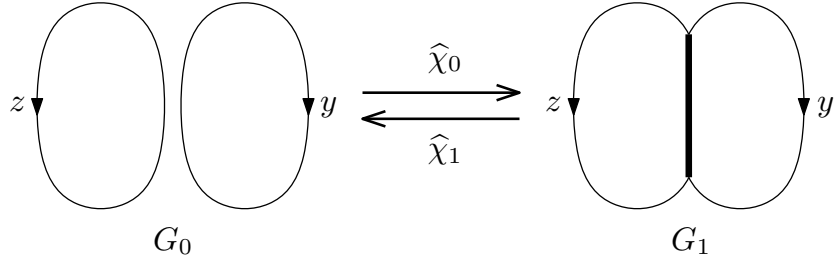}
 \caption{Local edge maps}
  \label{fig:edgemapsnew}
\end{figure}

\begin{figure}
 \centering
 \includegraphics[width=7.5cm]{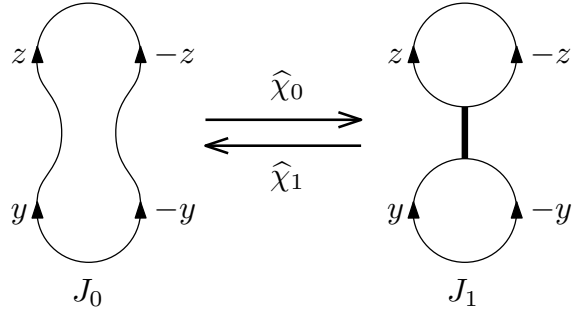}
 \caption{Local edge maps}
  \label{fig:edgemaps1}
\end{figure}

Consider first the two graphs $G_0$ and $G_1$ in Figure \ref{fig:edgemapsnew}.  After eliminating all factors corresponding to any unaffected wide edges and extra regular arcs, nearby markings will be relabeled as shown.  These graphs therefore correspond to matrix factorizations 
\[ 
C_2(G_0)=\left( \begin{array}{cc}
3y^2 & 0 \\
3z^2 & 0
\end{array} \right)_{\mathbb{Q}[y,z]}
\qquad \text{and} \qquad
C_2(G_1)=\left( \begin{array}{cc}
-3(y+z) & 0 \\
3(y^2+yz+z^2) & 0
\end{array} \right)_{\mathbb{Q}[y,z]}.
\]
As a $\mathbb{Q}[y,z]$-modules, $H_2(G_0)$ is generated by the element $|11\rangle_{G_0}$, while $H_2(G_1)$ is generated by $|11\rangle_{G_1}$.  Then it is easy to check that
\[
\widehat{\chi}_0 (|11\rangle_{G_0}) = |11\rangle_{G_1} \qquad \text{and} \qquad \widehat{\chi}_1 (|11\rangle_{G_1}) = (y-z)|11\rangle_{G_0}.
\]

If $J_0$ and $J_1$ are the two graphs in Figure \ref{fig:edgemaps1}, then after eliminating all the unaffected factors and relabeling marked points we obtain matrix factorizations
\[ 
C_2(J_0)=\left( \begin{array}{cc}
y^2+yz+z^2 & z-y \\
y^2+yz+z^2 & y-z
\end{array} \right)_{\mathbb{Q}[y,z]}
\qquad \text{and} \qquad
C_2(J_1)=\left( \begin{array}{cc}
0 & y^2-z^2 \\
3z^2 & 0
\end{array} \right)_{\mathbb{Q}[y,z]}.
\]
Then $H_2(J_0)$ is generated as a $\mathbb{Q}[y,z]$-module by $|01\rangle_{J_0}+|10\rangle_{J_0}$, while $H_2(J_1)$ is generated by $|01\rangle_{J_1}$.  Then 
\[
\widehat{\chi}_0 (|01\rangle_{J_0}+|10\rangle_{J_0}) = (y+z)|01\rangle_{J_1} \qquad \text{and} \qquad \widehat{\chi}_1 (|01\rangle_{J_1}) = |01\rangle_{J_0}+|10\rangle_{J_0}.
\]

\begin{rem}
The above computations are dependent on the planar configurations of the diagrams in Figures \ref{fig:edgemapsnew} and \ref{fig:edgemaps1}.  In particular, they do not hold if we allow virtual crossings in our link diagrams.  The isomorphism we describe below is dependent on these computations, and hence breaks down in the case of virtual knots and links.     
\end{rem}

\section{Identifying the cubes of resolutions $U$ and $U_2$} 
\label{sec:identifyingcubes}

To prove Theorem \ref{maintheorem}, we will show that the cube $U_2$ from Section \ref{sec:KRsl2} is isomorphic to the ordinary Khovanov cube of resolutions $U$ as described in Section \ref{sec:ordinary}.  Let $D$ be a diagram with choice of an ordering on its $n = n_+ + n_-$ crossings, and a collection of marked points as above.  Assume that the positive crossings are listed first in the chosen ordering, followed by the negative crossings.  We prove the following:

\begin{prop}
\label{prop:identifyingcubes}
There is a homeomorphism $\sigma: \left[ 0,1 \right]^n \rightarrow \left[ 0,1\right]^{n_+}\times \left[-1,0\right]^{n_-}$ sending vertices to vertices and edges to edges, and a collection of isomorphisms 
\[
\left\{\theta_v:\mathbb{V}_v\rightarrow H_2(\Gamma_{\sigma (v)}) \text{ } | \text{ } v \text{ a vertex of } \left[0,1\right]^n\right\},
\]
such that $\psi_{\sigma (e)}\circ \theta_{v}=\theta_{v'}\circ\varphi_e$ whenever $e$ is an edge from $v$ to $v'$. Moreover, the $\theta_v$ can be chosen so that they preserve the homological grading, while sending elements of quantum degree $j$ in $\mathbb{V}_v$ to quantum degree $-j$ in $H_2(\Gamma_{\sigma (v)})$.
\end{prop}

We will construct $\sigma$, the isomorphisms $\theta_v$, and prove Proposition \ref{prop:identifyingcubes} in the subsections which follow.  First however, note that Theorem \ref{maintheorem} follows from Proposition \ref{prop:identifyingcubes}.  Indeed, the differential of the ordinary Khovanov complex $\mathscr{C}(D)$ was constructed as a signed sum of the edge maps $\varphi_e$, with signs $\epsilon_e$ chosen subject only to the condition that $d^2=0$.  We can thus set $\epsilon_e = \varepsilon_{\sigma (e)}$ for each edge $e$ of $\left[ 0,1 \right]^n$ (recall that the $\varepsilon_e = \pm1$ are the signs coming from the Leibniz rule in the definition of $\mathscr{C}_2(D)$ in (\ref{eqn:chaincomplex})).  With signs chosen in this way, Proposition \ref{prop:identifyingcubes} then implies that $\mathscr{C}(D) \cong H(\mathscr{C}_2(D),d_{mf})$ as graded chain complexes, and hence Theorem \ref{maintheorem} follows.

\subsection{Identifying the underlying cubes} 
\label{subsec:underlyingcubes} We begin by defining the homeomorphism $\sigma : \left[ 0,1 \right]^n \longrightarrow \left[ 0,1\right]^{n_+}\times \left[-1,0\right]^{n_-}$.  Define $\sigma$ on the vertices $v$ of $\left[0,1\right]^n$ by 
\[
\sigma (v) = \sigma (v_1,\ldots,v_{n_+},v_{n_++1},\ldots,v_n) = (v_1,\ldots,v_{n_+},v_{n_++1}-1,\ldots,v_n-1),
\]
and extend $\sigma$ along the edges to a homeomorphism on the rest of $\left[0,1\right]^n$.

Recall that the $k$-th coordinate of $v \in U$ corresponds to the type of smoothing performed at the $k$-th crossing of $D$ to get $D_v$, while the $k$-th coordinate of $v' \in U_2$ corresponds to the type of resolution performed to the $k$-th crossing to obtain $\Gamma_{v'}$.  Thus positive crossings given 0 or 1-smoothings in $D_v$ are respectively given 0 or 1-resolutions in $\Gamma_{\sigma (v)}$, while negative crossings given 0 or 1-smoothings at $v$ are respectively given $-1$ or 0-resolutions in $\Gamma_{\sigma(v)}$.  

Passing from $D_v$ to $\Gamma_{\sigma (v)}$ can then be done simply by adding a wide edge at each 1-smoothed positive crossing and each 0-smoothed negative crossing, and adding the orientation induced from $D$.  Passing back from $\Gamma_{\sigma (v)}$ to $D_v$ is done by erasing all wide edges and forgetting orientations.  In particular, the components of $D_v$ will correspond to the components of $\widetilde{\Gamma}_{\sigma(v)}$ (which was obtained from $\Gamma_{\sigma(v)}$ by erasing all wide edges).

\subsection{Coherent choices of signs} 
\label{subsec:signs} In defining the $\theta_v$ as above, we must make a consistent choice of generators for the $R_v$-modules $H_2(\Gamma_v)$, which involves first making consistent choices of signs.

To do this, it is convenient for the time being to think of $D$ as an oriented planar graph, with vertices of degree 4 corresponding to crossings, and vertices of degree 2 corresponding to marked points.  In what follows we will consider paths $\gamma$ in this planar graph (which we continue to denote by $D$) which satisfy the following properties:
\begin{enumerate}
\item $\gamma$ starts and ends at 2-valent vertices of $D$,
\item each time $\gamma$ passes a 4-valent vertex of $D$, $\gamma$ switches from the overcrossing strand to the undercrossing strand or vice versa, 
\item \label{assumption:path} except for the start and end points of $\gamma$, each time $\gamma$ visits a 2-valent vertex it approaches and leaves along different edges, and
\item no edge is traversed more than once.
\end{enumerate}
In other words, the graphs we will be interested in start and end at marked points, turn left or right at vertices corresponding to crossings, travel straight through vertices corresponding to marked points, and do not repeat any edges.  

Suppose we fix an orientation on such a path $\gamma$, and suppose that at a 4-valent vertex $c$ the path $\gamma$ approaches along the edge $e$ and leaves along $e'$.  Note that $e$ and $e'$ are equipped with orientations both as edges of the oriented planar graph $D$, and as edges of the orientated path $\gamma$.  For such edges we will refer to the orientation from $D$ as the \emph{graph orientation}, and the orientation from $\gamma$ as the \emph{path orientation}. Now if the graph orientation on exactly one of the edges $e$ or $e'$ agrees with its path orientation, while on the other edge the two orientations disagree, then we say that $\gamma$ takes an \emph{orientation reversing turn} at $c$.  Clearly this does not depend on the orientation chosen on $\gamma$.  As we move along $\gamma$, note that the property of the path and graph orientations either agreeing or disagreeing will change precisely when the path takes an orientation reversing turn. For any such path $\gamma$, let $\ell(\gamma)$ be the number of orientation reversing turns taken by $\gamma$.
  
Note that if $\gamma$ is such a path between marked points $z$ and $z'$, then there is a vertex $v$ of the cube $U_2$ (non-unique in general) so that $\gamma$ gives rise to a path $\widehat{\gamma}$ in an obvious way between $z$ and $z'$ in the resolution $\Gamma_v$.  Each orientation reversing turn taken by $\gamma$ corresponds to a trivalent vertex of $\Gamma_v$ passed by $\widehat{\gamma}$.    
  
Now choose a single marked point on every component of the planar graph $D$ (note that these are \emph{not} in general the components of the underlying link $L$).  Denote these chosen marked points as $w_1, \ldots , w_h$.  Let $Z$ denote the set of marked points of $D$.  

Then we define a map $\tau: Z \rightarrow \{-1,1\}$ by setting $\tau (z) = (-1)^{\ell(\gamma_z)}$, where $\gamma_z$ is any path as above connecting the marked point $z$ to one of the $w_j$. 

\begin{lemma}
$\tau(z)$ depends only on the choice of the marked points $w_j$, and not on the choice of path $\gamma_z$.
\end{lemma}

\begin{proof}
First note there since there is only one $w_j$ on each component, the endpoints of any such path $\gamma_z$ are fixed.  Let $\gamma$ and $\gamma'$ be any two such paths connecting $z$ and the point $w_j$.  Orient both $\gamma$ and $\gamma'$ from $z$ to $w_j$.  The proof proceeds by considering three possible cases.

\emph{Case 1}. Suppose first that $\gamma$ and $\gamma'$ share both the same initial and final edges, which edges we label $e_0$ and $e_1$.  Then the path orientations induced on $e_0$ by $\gamma$ and $\gamma'$ must either both agree with or disagree with the graph orientation on $e_0$, and similarly with $e_1$.  Because the agreement or disagreement of path and graph orientations switches precisely at orientation reversing turns, this implies that $\ell(\gamma) \equiv \ell(\gamma') \text{ mod }2$. 

\emph{Case 2}. Suppose instead that both the initial and final edges of $\gamma$ and $\gamma'$ are different.  Then a similar argument again proves that $\ell(\gamma) \equiv \ell(\gamma') \text{ mod }2$. 

\emph{Case 3}.  Finally suppose that $\gamma$ and $\gamma'$ share either the same initial edges or final edges, but not both.  Assume for concreteness that they both have initial edge $e_0$.  

Since $\gamma$ and $\gamma'$ share the same initial edge, the two paths will agree until they diverge at some 4-valent vertex $c$, one path turning left while the other turns right.  Assume first that after diverging at $c$ the paths remain disjoint until they meet again at their common endpoint $w_j$.  Note however that $\gamma$ and $\gamma'$ may still have self-intersections of the form shown in the left-hand-side of Figure \ref{fig:localgraphmod} (where $\gamma$ may be replaced with $\gamma'$, and other choices of path orientations may be substituted). 

Now in a neighborhood of each of these self-intersections we make a local modification to $D$ as shown in Figure \ref{fig:localgraphmod}, and call the resulting graph $\widetilde{D}$.  Note that $\gamma$ and $\gamma'$ give rise to paths $\widetilde{\gamma}$ and $\widetilde{\gamma}'$ in $\widetilde{D}$ in the obvious way.
\begin{figure}
 \centering
 \includegraphics[width=11cm]{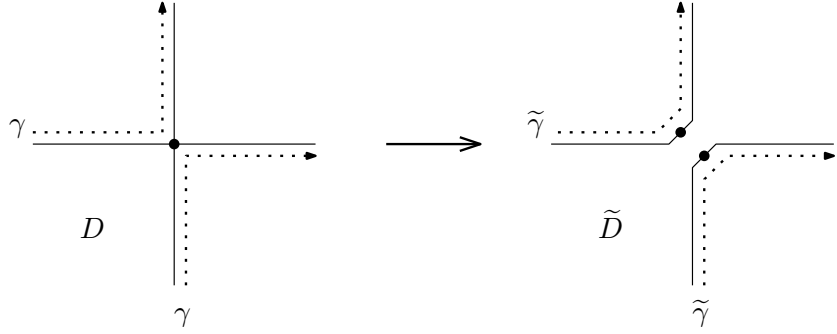}
  \caption{Self-intersections of $\gamma$ and $\gamma'$ with local resolution}
  \label{fig:localgraphmod}
\end{figure}
Consider the simple closed curve $\alpha \in \mathbb{R}^2$ obtained by starting at $c$, traveling along $\widetilde{\gamma}$ to $w_j$, and then returning to $c$ along $\widetilde{\gamma}'$, following its reverse orientation.  Let $A$ denote the compact region of $\mathbb{R}^2$ bounded by $\alpha$, and let $D'=\widetilde{D} \cap A$ be the graph consisting of all vertices and edges which lie entirely in $A$.

The vertex $c$ of $D$ will give rise to a 4-valent vertex of $\widetilde{D}$, which in turn yields a 3-valent vertex when we restrict to $D'$.  The other vertices of $\widetilde{D}$ along $\alpha$ will give either 2 or 4-valent vertices of $D'$.  Finally, all of the vertices of $\widetilde{D}$ lying on the interior of $A$ will have valence either 2 or 4, which will not change when we restrict to $A$.  Thus $D'$ is a graph with precisely one odd degree vertex, a contradiction. 

Now consider the general case where $\gamma$ and $\gamma'$ meet after they diverge at $c$ but before they reach $w_j$.  Suppose first that sometime after diverging the two paths again share an edge $e_2$ on which they induce the same path orientation.  Then by ignoring all of the edges after $e_2$ in $\gamma$ and $\gamma'$ and applying Case 1 shows that the number of orientation reversing turns taken by these two subpaths must agree mod 2.  Thus we can ignore any such loops, and assume that after diverging at $c$ the paths do not share an edge on which they induce the same path orientation.

Assume instead that after diverging the two paths share an edge $e_2$ on which they induce different path orientations.  Also assume that $e_2$ is the first such edge on $\gamma$ with this property.  By assumption (\ref{assumption:path}) on our paths $\gamma$ and $\gamma'$, they also must share at least one additional edge $e_3$ on which they induce different path orientations, which can be chosen so that it shares a 2-valent vertex with $e_2$.  This 2-valent vertex corresponds to a marked point, which we call $w$.  Following the orientation of $\gamma$ locally, we cross $e_2$ first to $w$, then cross the edge $e_3$.  Following $\gamma'$ instead reverses the order of $e_2$ and $e_3$.  By discarding the edges of $\gamma$ which come after $e_2$ and the edges of $\gamma'$ which come after $e_3$, we obtain subpaths of $\gamma$ and $\gamma'$ respectively from $z$ to $w$ which do not share any edges after they diverge at the vertex $c$.  Furthermore they still share the same initial edge $e_0$ but have different final edges $e_2$ and $e_3$.  We can thus assume that our paths $\gamma$ and $\gamma'$ are of this form.

Starting at $c$ and following $\gamma$ along its orientation, let $v$ be the first point of $\gamma$ which intersects $\gamma'$.  Since $\gamma$ and $\gamma'$ do not share any edges after the paths diverge at $c$, this intersection will necessarily look like one of the local pictures in Figure \ref{fig:localgraphintersection} (possibly after applying reflections and rotations).  Let $\beta$ be the path obtained by following $\gamma$ from $c$ to $v$, and then $\gamma'$ from $v$ back to $c$ along its reverse orientation.  
\begin{figure}
 \centering
 \includegraphics[width=10cm]{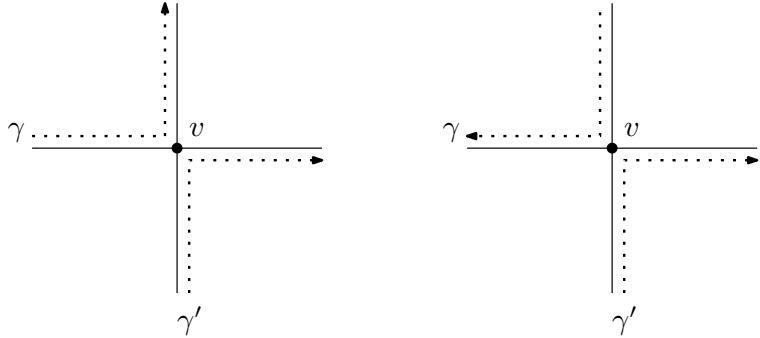}
  \caption{Two possible intersections types between $\gamma$ and $\gamma'$}
  \label{fig:localgraphintersection}
\end{figure}

Suppose first that the intersection at $v$ looks locally like the left-hand-side of Figure \ref{fig:localgraphintersection}.  By applying the modifications of Figure \ref{fig:localgraphmod} along $\beta$, we can again construct a counterexample as above. 

Thus we assume instead then that the intersection at $v$ looks locally like the right-hand-side of Figure \ref{fig:localgraphintersection}.  In this case the path $\beta$ necessarily separates the two arrow heads in this picture.  Since both paths end at $w$, and since the only possible intersections and self-intersections are as in Figures \ref{fig:localgraphmod} and \ref{fig:localgraphintersection}, this yields the last required contradiction and proves the lemma.
\end{proof}

Now suppose that a choice of $w_1, \ldots , w_h$ has been fixed for the marked diagram $D$, so that $\tau(z)$ is defined for all marked points $z$.  Let $v$ be any vertex of the cube $U_2$ of resolutions, with associated matrix factorization homology $H_2(\Gamma_v)$.  Recall that this homology was computed in Lemma \ref{lem:matrixfactorizationstructure}, where the matrix factorization $C_2(\Gamma_v)$ was originally considered as a module over the ring $R = \mathbb{Q}\left[z_1,\ldots,z_m\right]$.  As we eliminated the rows of $C_2(\Gamma_v)$, we made identifications of these $z_i$ until we were left with a matrix factorization over the ring $R_v$.  Note that if $z_i$ is a marked point on the same component of $\widetilde{\Gamma}_v$ as $z_j$ (recall that $\widetilde{\Gamma}_v$ is the result of removing all wide edges from $\Gamma_v$), then it is easy to verify that $z_i$ will be identified with $\tau(z_i)\tau(z_j)z_j$ in passing to this new ground ring.

\subsection{Choosing generators of $H_2(\Gamma_v)$}
\label{subsec:generators} The final remaining step before defining the vertex isomorphisms is that we must make choose generators of $H_2(\Gamma_v)$ for each vertex $v$ of $U_2$.  To do this, start with the vertex $v_0 = \{0,\ldots,0,-1,\ldots,-1\}$, where the first $n_+$ coordinates are 0 and the remaining $n_-$ are $-1$.  Let $g_{v_0}$ denote some choice of generator of $H_2(\Gamma_{v_0})$ as a $R_{v_0}$-module.  

We define the generators at the remaining vertices inductively.  Suppose that $v$ and $v'$ are two vertices, that $v'$ is obtained from $v$ by increasing one of its coordinates by $+1$, and that a generator $g_v$ has already been chosen for $H_2(\Gamma_v)$.  Let $e$ be the edge joining $v$ and $v'$.   

Suppose first that $e$ corresponds either to the left-to-right modification in Figure \ref{fig:edgemapsnew}, or the right-to-left modification in Figure \ref{fig:edgemaps1}.  Then we define $g_{v'} = \psi_e (g_v)$.  If, on the other hand, $e$ corresponds to either the right-to-left modification in Figure \ref{fig:edgemapsnew}, or the left-to-right modification in Figure \ref{fig:edgemaps1}, then we define
\[
g_{v'} = \dfrac{\psi_e (g_v)}{\tau(y)y+\tau(z)z}.
\]  
By our computations in Section \ref{sec:computingedgemaps} we see that $g_{v'}$ generates $H_2(\Gamma_{v'})$ as an $R_{v'}$-module (here we use the fact that in Figure \ref{fig:edgemapsnew} we have $\tau(z)=-\tau(y)$, whereas in Figure \ref{fig:edgemaps1} we have $\tau(z)=\tau(y)$).  It is straightforward to check that these $g_v \in H_2(\Gamma_v)$ are well-defined since the edge maps $\psi_e$ commute around the faces of $U_2$.

\subsection{The isomorphisms $\theta_v$}
\label{subsec:isomorphisms}  We are finally in a position to define the isomorphisms $\theta_v : \mathbb{V}_v \rightarrow H_2(\Gamma_{\sigma(v)})$.  Let $v$ be a vertex of $U$, and suppose that we have chosen an ordering on the $s = s(v)$ components of the total smoothing $D_v$.  Suppose also that we have labeled the marked points so that $z_j$ lies on the $j$th component of $D_v$ for $j=1,\ldots,s$.  Recall that $\mathbb{V}_v$ is the tensor product of $s$ copies of the 2-dimensional vector space generated by $\{1,x\}$, each one corresponding to a component of $D_v$.  In the factor corresponding to the $j$th component, we will denote the generator 1 by $1_j = x^0_j$ and the generator $x$ by $x_j = x^1_j$.  Then $\mathbb{V}_v$ has a $\mathbb{Q}$-basis 
\[
\{ x_1^{\eta_1} \otimes \cdots \otimes x_s^{\eta_s} \text{ }| \text{ } \eta_1,\ldots ,\eta_s \in \{0,1\} \},
\]
while as a $\mathbb{Q}$-vector space $H_2(\Gamma_{\sigma(v)}) \cong \mathbb{Q}[z_1,\ldots,z_s]/(z_1^2,\ldots,z_s^2)$ has basis
\[
\{ z_1^{\mu_1}\cdots z_s^{\mu_s} \text{ } | \text{ } \mu_1, \ldots ,\mu_s \in \{0,1\}\}.
\]
We then define the $\mathbb{Q}$-vector space isomorphism $\theta_v$ by
\[
\theta_v ( x_1^{\eta_1} \otimes \cdots \otimes x_s^{\eta_s} ) = \left(\prod_{i=1}^s (\tau(z_i)z_i)^{\eta_i}\right) g_{\sigma(v)}.
\] 

It is easy to verify that these isomorphisms commute with the edge maps.  For example, suppose that $v$ and $v'$ are vertices of $U$ joined by an edge $e$, which corresponds to the splitting of one component in $D_v$ into two.  For concreteness assume the that $s$th component of $D_v$ splits to create the $s$th and $(s+1)$th components of $D_{v'}$.  Then
\begin{align*}
\theta_{v'}\circ\varphi_e\left(x_1^{\eta_1} \otimes \cdots \otimes x^{\eta_{s-1}}_{s-1} \otimes 1_s\right) & = \theta_{v'}\left(x_1^{\eta_1} \otimes  \cdots \otimes x^{\eta_{s-1}}_{s-1} \otimes x_s \otimes 1_{s+1}\right) \\
&\qquad \qquad + \theta_{v'}\left(x_1^{\eta_1} \otimes \cdots \otimes x^{\eta_{s-1}}_{s-1} \otimes 1_s \otimes x_{s+1}\right) \\
& = \left(\prod_{i=1}^{s-1} (\tau(z_i)z_i)^{\eta_i}\right) \left( \tau(z_s)z_s+\tau(z_{s+1})z_{s+1} \right)g_{\sigma(v')},
\end{align*}   
while
\begin{align*}
\psi_{\sigma(e)}\circ\theta_v\left(x_1^{\eta_1} \otimes \cdots \otimes x^{\eta_{s-1}}_{s-1} \otimes 1_s\right) & = \psi_{\sigma(e)}\left( \left(\prod_{i=1}^{s-1} (\tau(z_i)z_i)^{\eta_i}\right) g_{\sigma(v)}\right)\\
& =\left(\prod_{i=1}^{s-1} (\tau(z_i)z_i)^{\eta_i}\right) \psi_{\sigma(e)}\left(g_{\sigma(v)}\right)\\
& =\left(\prod_{i=1}^{s-1} (\tau(z_i)z_i)^{\eta_i}\right)\left( \tau(z_s)z_s + \tau(z_{s+1})z_{s+1}\right)g_{\sigma(v')}.
\end{align*}
Likewise we have 
\begin{align*}
\theta_{v'}\circ\varphi_e\left(x_1^{\eta_1} \otimes \cdots \otimes x^{\eta_{s-1}}_{s-1} \otimes x_s\right) & = \theta_{v'}\left(x_1^{\eta_1} \otimes \cdots \otimes x^{\eta_{s-1}}_{s-1} \otimes x_s \otimes x_{s+1}\right) \\
& = \left(\prod_{i=1}^{s-1} (\tau(z_i)z_i)^{\eta_i}\right)  \tau(z_s)\tau(z_{s+1})z_sz_{s+1}g_{\sigma(v')},
\end{align*}   
while
\begin{align*}
\psi_{\sigma(e)}\circ\theta_v\left(x_1^{\eta_1} \otimes \cdots \otimes x^{\eta_{s-1}}_{s-1} \otimes x_s\right)& = \psi_{\sigma(e)}\left( \left(\prod_{i=1}^{s-1} (\tau(z_i)z_i)^{\eta_i}\right) \tau(z_s)z_sg_{\sigma(v)}\right)\\
& =\left(\prod_{i=1}^{s-1} (\tau(z_i)z_i)^{\eta_i}\right) \tau(z_s)z_s\psi_{\sigma(e)}\left(g_{\sigma(v)}\right)\\
& =\left(\prod_{i=1}^{s-1} (\tau(z_i)z_i)^{\eta_i}\right)\left( \tau(z_s)^2z^2_s+\tau(z_s)\tau(z_{s+1})z_sz_{s+1}\right)g_{\sigma(v')}\\
& =\left(\prod_{i=1}^{s-1} (\tau(z_i)z_i)^{\eta_i}\right)\tau(z_s)\tau(z_{s+1})z_sz_{s+1}g_{\sigma(v')},
\end{align*}
where for the last equality we used the fact that $z^2_s = 0$ in $H_2(\Gamma_{\sigma(v')})$.  The case of two circles of $D_v$ merging into one circle of $D_{v'}$ is similar, and is left to the reader.

\subsection{Homological gradings}
\label{subsec:homologicalgradings} It remains only to verify that the isomorphisms $\theta_v$ preserve the homological grading and change the quantum gradings by a sign.  Since the homological grading of any element in either setup depends only on the the vertex at which it lives, we can easily check that both $v$ and $\sigma (v)$ will have the same homological degrees (more precisely, the homological degree of any $y \in \mathbb{V}_v$ will equal the homological degree of any $y' \in H_2 (\Gamma_{\sigma (v)})$).  In the ordinary cube $U$, the homological degree at a vertex $v = (v_1, \ldots , v_n)$ is given by 
\begin{equation}
\label{eqn:homologicaldegree1}
r(v) - n_- = \sum_{i=1}^n v_i - n_-,
\end{equation} 
while at vertex $\sigma(v) = (v'_1, \ldots , v'_n)$ in the cube $U_2$, the homological degree is
\begin{equation}
\label{eqn:homologicaldegree2}
\sum_{i=1}^n v'_i = \sum_{i=1}^{n_+} v_i + \sum_{i=n_++1}^{n} (v_i-1) = \sum_{i=1}^n v_i - n_-.
\end{equation}

\subsection{Quantum gradings}
\label{subsec:quantumgradings} Recall now that each vertex $v'$ of the cube $\left[0,1\right]^{n_+} \times \left[-1,0\right]^{n_-}$ has an additional $q$-grading shift coming from the definitions of $\mathscr{C}_2(P_+)$ and $\mathscr{C}_2(P_-)$ in (\ref{eqn:P+}) and (\ref{eqn:P-}).  For a vertex $v' = (v'_1, \ldots , v'_n)$ the $q$-grading of the space $H_2(\Gamma_{v'})$ is additionally shifted by 
\begin{equation}
\label{eqn:gradingshift}
p_{v'}=\sum_{i=1}^{n_+} -(v'_i + 1) + \sum_{i=n_++1}^n (1-v'_i) = -\sum_{i=1}^n v'_i + n_--n_+.
\end{equation}
Now if $v=\sigma^{-1}(v')$, then the number of circles obtained from deleting all the wide edges of $\Gamma_{v'}$ will be equal to the number of components of $D_v$, namely $s(v)$.  Combining (\ref{eqn:homologicaldegree1}), (\ref{eqn:homologicaldegree2}), (\ref{eqn:gradingshift}), and Lemma \ref{lem:matrixfactorizationstructure} then gives  
\[
H_2 (\Gamma_{v'})\{p_{v'}\} \cong \left(\mathbb{Q} \left[z_1, \dots , z_k\right]/(z_1^2, \ldots , z_k^2)\right) \{-s(v)-r(v) -n_+ + 2n_-\}.
\]

Consider the element $x_1^{\eta_1} \otimes \cdots \otimes x_s^{\eta_s} \in \mathbb{V}_v$.  Recall that $\text{deg }1=1$ while $\text{deg }x=-1$, and that there is an additional degree shift of $r(v)+n_+-2n_-$ in the definition of $\mathbb{V}_v$.  Then 
\[
\text{deg }x_1^{\eta_1} \otimes \cdots \otimes x_s^{\eta_s} = s(v)-2\sum_{i=1}^s\eta_i+r(v)+n_+-2n_-.
\]
But $\text{deg }g_{\sigma(v)} = -s(v)-r(v) -n_+ + 2n_-$, while $\text{deg }z_j = 2$ for each $z_j$.  This shows that 
\[
\text{deg }\theta_v \left(x_1^{\eta_1} \otimes \cdots \otimes x_s^{\eta_s}\right)= -s(v)+2\sum_{i=1}^s \eta_i -r(v) -n_+ + 2n_-,
\]
which completes the proofs of Proposition \ref{prop:identifyingcubes} and Theorem \ref{maintheorem}.

\nocite{WuKR}
\bibliographystyle{amsplain}
\bibliography{bibliography}

\end{document}